\documentclass[11pt]{article}

\usepackage{amsmath,amsthm,amssymb}

\usepackage{hyperref}

\usepackage[utf8]{inputenc}

\usepackage{gensymb}

\usepackage{tipa}

\usepackage{thmtools}
\usepackage{thm-restate}

\theoremstyle{plain}
\newtheorem{theorem}{Theorem}
\newtheorem{proposition}{Proposition}
\newtheorem{lemma}[theorem]{Lemma}

\theoremstyle{definition}
\newtheorem*{definition}{Definition} 

\theoremstyle{remark}

\theoremstyle{question}
\newtheorem{question}{Question}

\usepackage[margin=1in]{geometry}

\usepackage{verbatim}

\newcommand{\sz}[2]{S_{\mathfrak{z}}(#1,#2)}
\newcommand{\szt}[2]{S_{\mathfrak{z},2}(#1,#2)}
\newcommand{\nequiv}{\not\equiv}
\newcommand{\lamba}{\lambda}
\newcommand{\lamda}{\lambda}

\newcommand{\ssq}{$\star$-sequence }

\begin{document}

\title{Upper and Lower Bounds on Zero-Sum Generalized Schur Numbers}
\author{Erik Metz}
\maketitle

\begin{abstract}
Let $\sz{k}{r}$ be the least positive integer such that for any $r$-coloring $\chi : \{1,2,\dots,\sz{k}{r}\} \longrightarrow \{1, 2, \dots, r\}$, there is a sequence $x_1, x_2, \dots, x_k$ such that $\sum_{i=1}^{k-1} x_i = x_k$, and $\sum_{i=1}^{k} \chi(x_i) \equiv 0 \pmod{r}$. We show that when $k$ is greater than $r$, $kr - r - 1 \le \sz{k}{r} \le kr - 1$, and when $r$ is an odd prime, $\sz{k}{r}$ is in fact equal to $kr - r$. 
\end{abstract}

\section{Introduction}
The generalized Schur numbers $S(k,r)$ are an object in Ramsey theory defined to be the least positive integer such that any $r$-coloring of $\{1,2,\dots,\sz{k}{r}\}$ admits a monochromatic solution to $\sum_{i=1}^{k-1} x_i = x_k$. In 1916, Schur proved that $\frac{1}{2}(3^n - 1) \le S(3,n) \le R(n) - 2$, where $R(n)$ is the $n$th diagonal Ramsey number~\cite{schur}. 
The lower bound has since been improved to $S(3,n) > c(3.17176)^n$ for some positive constant $c$ by Exoo in 1994~\cite{exoo}. However, since the best upper bounds on diagonal Ramsey numbers are still $\omega((4-\epsilon)^{n})$, there is a very large gap between the lower and upper bounds on Schur numbers.

In 2018 Robertson 
 introduced the zero-sum generalized Schur numbers, which relax the monochromatic condition to a zero-sum condition:
 
 \begin{definition}[\cite{robertson}]
 We call a sequence $x_1, x_2, \dots, x_k$  \emph{$r$-zero-sum} if $\sum_{i=1}^{k} x_i \equiv 0 \pmod{r}$. 
 \end{definition}
 
 A fundamental result in the study of zero-sum sequences is the 1961 Erd\H{o}s-Ginzberg-Ziv theorem, which states that in any set of $2n - 1$ integers, there is a subset of size $n$ which forms an $n$-zero-sum sequence~\cite{egz}. Since then, zero-sum problems have been a fruitful area of research in both additive number theory and Ramsey theory (see~\cite{zerosum} for many examples). More recently, authors have been studying zero-sum sequence problems with more rigid structures imposed upon the sequence (\cite{robertson}, \cite{robertsonvdw}, \cite{brown}). In this paper we will impose the same structure as in the generalized Schur numbers, that the sum of the first $k-1$ terms of the sequence is equal to the final term.
 
 \begin{definition}
 We denote by $\mathcal{E}$ the equation $\sum_{i=1}^{k-1} x_i = x_k$.
 \end{definition}
 
 \begin{definition}
 The \emph{zero-sum generalized Schur number} $\sz{k}{r}$ is the least positive integer such that for any $r$-coloring $\chi : \{1,2,\dots,\sz{k}{r}\} \longrightarrow \{1, 2, \dots, r\}$, there is a solution $x_1, x_2, \dots, x_k$ to $\mathcal{E}$ such that $\chi(x_1), \chi(x_2), \dots, \chi(x_k)$ is an $r$-zero-sum sequence.
 \end{definition}
 
 Note that if $r \nmid k$, then $\sz{k}{r} = \infty$, since there will be no $r$-zero-sum solutions to $\mathcal{E}$ if we color everything with $1$. If $r \mid k$, then any monochromatic solution to $\mathcal{E}$ is automatically $r$-zero-sum, so $\sz{k}{r} \le S(k,r)$.

In 2018, Robertson proved that $\sz{k,2} = 2k - 3$, and also discovered lower bounds on $\sz{k}{r}$ when $r = 2$ or $3$~\cite{robertson}. 
In particularly, Robertson proved the following theorems.
\begin{theorem}[\cite{robertson}, Theorem 4]
Let $k \in \mathbb{Z}^+$ with $3 \mid k$. Then $\sz{k}{3} \ge 3k - 3$. 
\end{theorem}
\begin{theorem}[\cite{robertson}, Theorem 6]
Let $k \in \mathbb{Z}^+$ with $4 \mid k$. Then $\sz{k}{4} \ge 4k - 5$.
\end{theorem}
Robertson also provided a proposition in the case where $k = r$.
\begin{proposition}[\cite{robertson}, Proposition 7]
Let $k$ be an odd positive integer. Then $\sz{k}{k} \ge 2(k^2  - k - 1)$.
\end{proposition}

These results naturally led Robertson to pose the following four questions regarding $\sz{k}{r}$. 
\begin{question}
	 Does $\sz{k}{3} = 3k - 3$ for $3 \mid k$, $k \ge 6$?
\end{question}
\begin{question}
 Does $\sz{k}{4} = 4k - 5$ for $4 \mid k$, $k \ge 8$?
\end{question}
\begin{question}
 What is the exact value of $\szt{k}{4}$? ($\szt{k}{r}$ is the same as $\sz{k}{r}$, except that the coloring only uses the colors $0$ and $1$ rather than any value from $1$ to $r$.)
\end{question}
\begin{question}
 Is it true that $\sz{k}{k}$ is of the order $k^2$?
\end{question}
Question 3 was answered by Robertson, Roy, and Sarkar with the following theorem.
\begin{theorem}[\cite{robertson2}, Theorem 3]
Let $k$ and $r$ be positive integers such that $r \mid k$ and $k > r$. Then, $\szt{k}{r} = rk - 2r + 1$.
\end{theorem}
We resolve the first two questions below, determine the exact value of $\sz{k}{r}$ in the case of $r$ prime and $k > r$, and provide fairly close (within $r$) bounds in other cases where $k > r$.

In particular, we prove the following upper bounds:

\begin{restatable}{theorem}{uprime}
\label{thm:uprime}
Let $r$ be an odd prime and let $k$ be a multiple of $r$ which is at least $2r$. Then $\sz{k}{r} \le kr - r$.
\end{restatable}

\begin{restatable}{theorem}{ufour}
\label{thm:ufour}
Let $k$ be a multiple of $4$ which is at least $8$. Then $\sz{k}{r} \le 4k - 5$.
\end{restatable}

\begin{restatable}{theorem}{ucomp}
\label{thm:ucomp}
Let $r \mid k$, $k \ge 2r$, and $r \ge 6$. Let the prime factors of $r$ be $p_0, p_1, \dots, p_{t_1}$ (the $p_i$ need not be distinct).Then $\sz{k}{r} \le kr - \sum_{i=0}^{t-1}(p_i - 1) - 1$. 
\end{restatable}

The last theorem shows that whenever $k$ is a multiple of $r$ which is greater than $r$, $\sz{k}{r} \le kr - 1$.
We also prove two lower bound theorems:
\begin{restatable}{theorem}{lodd}
\label{thm:lodd}
Suppose $r$ is odd. Then $\sz{k}{r} \ge kr - r$.
\end{restatable}
\begin{restatable}{theorem}{leven}
\label{thm:leven}
Suppose $r$ is even. Then $\sz{k}{r} \ge kr - r - 1$.
\end{restatable}
As a corollary, this shows that if $r$ is an odd prime and $k$ is a multiple of $r$ which is greater than $r$, then $\sz{k}{r} = kr - r$. Furthermore, if $r$ is any number and $k$ is a multiple of $r$ which is greater than $r$, then $kr - r - 1 \le \sz{k}{r} \le kr - 1$. Note that while the regular generalized Schur numbers had bounds 
with different exponentials even in the simplest case of $k = 3$, it was possible to find similar upper and lower bounds for every case of the zero-sum generalized Schur numbers. We would be interested to know how the bounds given in this paper can be improved, and also whether it is possible to find close upper and lower bounds in other zero-sum Ramsey type problems.

\section{Upper Bounds}

While the result for composite $r$ in Theorem~\ref{thm:ucomp} supersedes the results for prime $r$ in Theorem~\ref{thm:uprime}, we provide separate proofs of each for greater clarity. We begin by restating and proving Theorem~\ref{thm:uprime}.

\uprime*

\begin{proof}
Let $\chi: \{1,\dots,kr - r\} \longrightarrow \{1, \dots, r \}$ be an $r$-coloring. We will show by contradiction that $\chi$ admits a zero-sum solution to $\mathcal{E}$. So, suppose that there is no such solution.
First, we claim that 
\begin{equation} \label{recurrence}
2 \chi(\alpha) \equiv \chi(\alpha - 1) + \chi(\alpha + 1) \pmod{r}
\end{equation} for $2 \leq \alpha \leq r-1$.

To show this, let $s$ be any integer between $0$ and $r - 1$. Then the following is a solution to $\mathcal{E}$:
$$(k - 1 - 2s)\cdot \alpha + s \cdot (\alpha - 1) + s \cdot (\alpha + 1) = (k-1)\alpha.$$
Here $a \cdot x$ denotes the sum of a sequence of $a$ copies of $x$. 
Hence,
$$(k - 1 - 2s)\chi(\alpha) + s(\alpha - 1) + s(\alpha + 1) + \chi((k-1)\alpha) \nequiv 0 \pmod{r},$$
which can be rewritten as
$$(k - 1)\chi(\alpha) + \chi((k-1)\alpha) \nequiv s(2\chi(\alpha) - \chi(\alpha - 1) - \chi(\alpha + 1)).$$
Since $r$ is prime and $s$ can take on any value modulo $r$, this implies that $2 \chi(\alpha) \equiv \chi(\alpha - 1) + \chi(\alpha + 1) \pmod{r}$, as desired.
Since $r \mid k$, applying global translations to $\chi$ does not affect which sequences have zero-sum images under $\chi$, so without loss of generality we may translate $\chi$ so that $\chi(2) \equiv 2\chi(1) \pmod{r}$.
Then by \eqref{recurrence} we have 
	\begin{equation} \label{formula}
	\chi(\alpha) \equiv \alpha \chi(1) \pmod{r}
	\end{equation}
for $\alpha \leq r$.
Let $m = kr - r.$ We will show that $\chi(k) \equiv 0 \pmod{r}$ by showing that
		\begin{equation} \label{inequality}
		-\chi(m) \nequiv s \chi(k) \pmod{r}
		\end{equation}
for $0 \leq s \leq r-1$.

Because $s \le r - 1$, we have $k - 1 - s \le m - sk \le r(k - 1 - s)$, so there exists a sequence $\alpha_1, \dots, \alpha_{k-1-s}$ of integers such that $1 \le \alpha_i \le r$ for all $i$ and $\alpha_1 + \dots + \alpha_{k-1-s} = m - sk$.
Then
$$s\cdot k + \alpha_1 + \dots + \alpha_{k - 1 - s}$$
is a solution to $\mathcal{E}$, so
$$s\chi(k) + \chi(\alpha_1) + \dots + \chi(\alpha_{k - 1- s}) + \chi(m) \nequiv 0 \pmod{r}.$$
By~\ref{formula}, this implies
$$-\chi(m) \nequiv s\chi(k),$$
as desired. So $\chi(k) \equiv 0 \pmod{r}$.
But then
$$(k - 2)\cdot 1 + 2 = k,$$
and 
$$(k-2)\chi(1) + \chi(2) + \chi(k) \equiv k\chi(1) + \chi(k) \equiv 0 \pmod{r},$$
contradicting our assumption that there is no zero-sum solution to $\mathcal{E}$.
\end{proof}

We now exhibit a more complicated proof of Theorem~\ref{thm:ufour}, the case $r = 4$. Although the basic structure of the proof is similar to to Theorem~\ref{thm:uprime}, we can no longer use the fact that if $r$ is prime then all numbers are either multiples of $r$ or relatively prime to $r$.

\ufour*
\begin{proof}
We proceed by contradiction. Suppose that there is a function $\chi: \{1, \dots, 4k-5\} \rightarrow \{1,2,3,4\}$ such that there are no 4-zero-sum solutions to $\mathcal{E}$. First we will show that $\chi(1), \chi(2), \chi(3),$ and $\chi(4)$ are all distinct.

Without loss of generality, assume $\chi(2) \equiv 2$. Note that $k-1 \ge 2(r-1)$, so for $0 \le s \le 3$ we have
\begin{align*}
&(k-1 - 2s)\cdot 2 + s \cdot 1 + s \cdot 3  = 2k - 2 \\
&\implies (k-1-2s)\chi(2) + s(\chi(1) + \chi(3)) + \chi(2k-2) \nequiv 0 \pmod{4}\\
&\implies \chi(2k-2) \nequiv -(k-1)\chi(2) - s(\chi(1) + \chi(3) - 2\chi(2)) \pmod{4}\\
&\implies \chi(2k-2) \nequiv 2 - s(\chi(1) + \chi(3)) \pmod{4}.
\end{align*}
If $\chi(1) + \chi(3)$ is relatively prime to $4$, no value can be assigned to $\chi(2k-2)$. So $\chi(1) + \chi(3)$ must be even.
Suppose $\chi(1) + \chi(3) \equiv 2$. Then $\chi(2k - 2)$ is odd, and we have for $s \in \{0,1\}$
\begin{align*}
&(2k-2) + (k-2 - 2s)\cdot 2 + s \cdot 1 + s \cdot 3 = 4k-6 \\
&\implies \chi(2k-2) + (k-2)\chi(2) + s(\chi(1) + \chi(3) - 2\chi(2)) + \chi(4k-6) \nequiv 0 \pmod{4}\\
&\implies \chi(4k-6) \nequiv -\chi(2k-2) - 2s \pmod{4}\\
&\implies \chi(4k-6) \text{ is even.}
\end{align*}
Now, again for $s \in \{0,1\}$ we have
\begin{align*}
&2\cdot k + (k - 3 - 2s) \cdot 2 + s \cdot 1 + s \cdot 3 = 4k - 6 \\
&\implies 2\chi(k) + (k - 3)\chi(2) + s(\chi(1) + \chi(3) - 2\chi(2)) + \chi(4k-6) \nequiv 0 \pmod{4}\\
&\implies 2\chi(k) + 2 + 2s + \chi(4k-6) \nequiv 0 \pmod{4}.
\end{align*}
But then $2\chi(k)$ must be odd, so this is a contradiction and $\chi(1) + \chi(3) \equiv 2\chi(2) \equiv 0$.

Now assume that $\chi(1)$ is even; that is, $2\chi(1) \equiv 0$. Note that this implies $\chi(3) \equiv \chi(1)$.
We can compute for $0 \le s \le r - 1$
\begin{align*}
&(k-1 - 2s)\cdot 3 + s \cdot 2 + s \cdot 4 = 3k - 3 \\
&\implies \chi(3k-3) + (k-1)\chi(3) + s(\chi(2) + \chi(4) - 2\chi(3)) \nequiv 0 \pmod{4} \\ 
&\implies \chi(3k-3) \nequiv \chi(3) - s(\chi(2) + \chi(4) - 2\chi(3)) \pmod{4} \\
&\implies \chi(3k-3) \nequiv \chi(3) - s(2 + \chi(4)) \pmod{4}.
\end{align*}
If $\chi(4)$ were odd, then no value could be assigned to $\chi(3k-3)$. Thus $\chi(4)$ is even.
Let $n(\star)$ denote a sequence of values from $1$ to $4$ summing to $n$. Now, for $m \in \{k, k+2, k+4, \dots, 4k - 6\}$, we can write $m$ as a sum of $1,2,3,$ and $4$, with an even number of each except possibly $2$. This is possible because we can assume there are not both $1$'s and $3$'s, and then by parity there must be an even number of whichever is there. Similarly if there are $4$'s we can assume there are no $1$'s. So the only possible bad case has an odd number of $4$'s, and even number of $3$'s, and an even number of $2$'s. If there is at least one $2$, exchange a $4$ and a $2$ for two $3$'s. If there are no $2$'s, then exchange two $3$'s for a $4$ and a $2$. If there are no $2$'s and at most one $3$, then the total sum is at least $4k - 5$, which is larger than we are concerned with.

So, for $m \in \{k, k+2, k+4, \dots, 4k-6\}$, there exist integers $a_i$ such that
$$m = 2a_1 \cdot 1 + 2a_2 \cdot 2 + 2a_3 \cdot 3 + 2a_4 \cdot 4 + 2,$$
which implies
$$\chi(m) + 2a_1\chi(1) + 2a_2\chi(2) + 2a_3\chi(3) + 2a_4\chi(4) + \chi(2) \nequiv 0 \pmod{4}.$$
Under our current assumptions, the colors of 1, 2, 3, and 4 are all even, and $\chi(2) \equiv 2$.
Hence,
	$$\chi(m) \nequiv 2 \pmod{4}.$$
Now, for $m \in \{2k - 2, 2k, \dots, 4k - 6\}$, there exist integers $a_i$ such that
$$m = k + 2a_1 \cdot 1 + 2a_2 \cdot 2  + 2a_3 \cdot 3 + 2a_4 \cdot 4.$$
Therefore. 
$$\chi(m) + \chi(k) + 2a_1\chi(1) + 2a_2\chi(2) + 2a_3\chi(3) + 2a_4\chi(4) \nequiv 0 \pmod{4},$$
which under our current assumptions implies 
$$\chi(m) \nequiv -\chi(k) \pmod{4}.$$

Similarly, for $m \in \{3k - 4, 3k - 2, \dots, 4k - 6 \}$, there exist integers $a_i$ such that
$$m = 2 \cdot k + 2a_1 \cdot 1 + 2a_2 \cdot 2 + 2a_3 \cdot 3 + 2a_4 \cdot 4 + 2,$$
which implies 
$$\chi(m) \nequiv -2\chi(k) + 2 \pmod{4}.$$

Finally, 
$$4k - 6 = 3\cdot k + (k-4)\cdot 1,$$
so 
$$\chi(4k - 6) \nequiv \chi(k) \pmod{4}.$$

If $\chi(k)$ is odd, no value may be assigned to $\chi(4k - 6)$. So we may assume $\chi(k)$ is even, and therefore $\chi(k) \equiv 0$.
Then for $m \in \{2k - 2, 2k, \dots, 4k - 6\}$, $\chi(m)$ must be odd.
Let $s$ be an even integer less than $k$, and suppose that $\chi(s)$ is odd. Note that $s$ is at least $6$.
Then there exist integers $a_i$ such that 
$$4k - 6 = 3\cdot s + 2a_1\cdot 1 + 2a_2\cdot 2 + 2a_3\cdot 3 + 2a_4\cdot 4,$$
which implies
$$\chi(4k - 6) \nequiv \chi(s) \pmod{4}.$$
On the other hand, there exist integers $b_i$ such that
$$4k - 6 = s + 2b_1\cdot 1 + 2b_2\cdot 2 + 2b_3\cdot 3 + 2b_4\cdot 4,$$
which implies
$$\chi(4k - 6) \nequiv -\chi(s) \pmod{4}.$$
So $\chi(4k - 6) \nequiv \pm \chi(s) \pmod{4}.$. Since $\chi(s)$ was assumed to be odd, this means that $\chi(4k - 6)$ must be even, a contradiction. Therefore for all even integers $s$ which are less than $k$, $\chi(s)$ is also even. 

Now suppose $t$ is an odd integer less than $k$, and $\chi(t)$ is odd. Note $t \ge 5$. Then there exist integers $a_i$ such that
$$4k - 6 = 3\cdot t + 1 + 2 + 2a_1 \cdot 1 + 2a_2\cdot 2 + 2a_3 \cdot 3 + 2a_4 \cdot 4,$$
which implies that
$$\chi(4k - 6) \nequiv \chi(t) - \chi(1) + 2\pmod{4}.$$
Similarly, there exist integers $b_i$ such that
$$4k - 6 = t + 1 + 2 +2b_1 \cdot 1 + 2b_2\cdot 2 + 2b_3 \cdot 3 + 2b_4 \cdot 4.$$
Hence,
$$\chi(4k - 6)\nequiv -\chi(t) - \chi(1) + 2 \pmod{4}.$$
Since $\chi(1)$ is assumed to be even, this shows that $\chi(4k - 6)$ cannot be odd, which is a contradiction.
Thus $\chi(s)$ is even for all integers $s$ less than $k$.

Now suppose $s$ is even, $k < s < 2k - 2$, and $\chi(s)$ is odd. Recall that $\chi(k) \equiv 0$. Let $\sigma = s + k + (k - 4)\cdot 1  + 2$.
Note that $3k - 2 < \sigma < 4k - 4$, and $\sigma$ is even.
We have 
$$s + \left(\frac{k}{2} - 2\right)\cdot 1 + \left(\frac{k}{2}\right) \cdot 3 = \sigma.$$
Therefore
$$\chi(s) + \left(\frac{k}{2} - 2\right)\chi(1) + \left(\frac{k}{2}\right)\chi(3) + \chi(\sigma) \nequiv 0 \pmod{4}.$$
The above expression can be simplified to $\chi(\sigma)\nequiv -\chi(s) \pmod{4}.$
On the other hand, $\sigma = s + k + (k - 4)\cdot 1  + 2$ implies that
$$\chi(s) + \chi(k)+ (k - 4)\chi(1) + \chi(2) + \chi(\sigma) \nequiv 0 \pmod{4},$$
which simplifies to $\chi(\sigma) \nequiv -\chi(s) - \chi(k)$.
But $\chi(k) \equiv 2 $ and $\chi(sigma)$ is odd, so this is impossible.
Similarly, if $s$ is odd, $k < s < 2k - 2$, and $\chi(s)$ is odd, we can define $\sigma = s + k + (k - 3) \cdot 1$.
Then $\sigma$ is even and $3k - 3 < \sigma < 4k - 5$.
We have
$$s + \left(\frac{k}{2} - 1\right)\cdot 1 + \left(\frac{k}{2} - 2\right) \cdot 3 + 2\cdot 2 = \sigma,$$
so 
$$\chi(s) +  \left(\frac{k}{2} - 1\right)\chi(1) + \left(\frac{k}{2} - 2\right)\chi(3) + 2\chi(2) + \chi(\sigma) \nequiv 0 \pmod{4}.$$
Since $\chi(1) \equiv \chi(3),$ this simplifies to $\chi(\sigma) \nequiv -\chi(s) + \chi(1) + \pmod{4}.$
On the other hand, we can write
$$s + k + (k - 3)\cdot 1 = \sigma,$$
so 
$$\chi(s) + \chi(k) + (k-3)\chi(1) + \chi(\sigma) \nequiv 0 \pmod{4}.$$
This simplifies to $\chi(\sigma) \nequiv -\chi(s) + \chi(1) + 2 \pmod{4}.$
This is a contradiction since $\chi(\sigma)$ and $\chi(s)$ are both odd, but $\chi(1)$ is even.
Thus $\chi(s)$ is even for $s \le 2k - 3$.
But $\sz{k}{2} = 2k - 3$ (\cite{robertson}, Proposition 3), so there is a $4$-zero-sum solution to $\mathcal{E}$ in the first $2k - 3$ values. This contradicts our assumption that $\chi(1)$ is even, and therefore we may now assume that $\chi(1)$ is odd. Since $-1$ is relatively prime to $4$, zero-sum solutions to $\mathcal{E}$ are invariant under multiplication of $\chi$ by $-1$. So, without loss of generality, we can assume that $\chi(1) \equiv 1$ Then $\chi(3) = 3$. By the same reasoning which showed $\chi(1) + \chi(3) - 2\chi(2)$ to be even, 
we can show that $\chi(2) + \chi(4) - 2\chi(3)$ is even.
For integers $s$ between 0 and 3 we have
\begin{align*}
&(k-1 - 2s)\cdot 3 + s \cdot 2 + s \cdot 4  = 3k - 3 \\
&\implies (k-1-2s)\chi(3) + s(\chi(2) + \chi(4)) + \chi(3k - 3) \nequiv 0 \pmod{4}\\
&\implies \chi(3k-3) \nequiv -(k-1)\chi(3) - s(\chi(2) + \chi(4) - 2\chi(3)) \pmod{4}\\
&\implies \chi(3k-3) \nequiv \chi(3) - s(\chi(2) + \chi(4) - 2\chi(3)) \pmod{4}.
\end{align*}
If $\chi(2) + \chi(4) - 2\chi(3)$ were odd, then no value could be assigned to $\chi(3k - 3)$. Therefore $\chi(2) + \chi(4) - 2\chi(3)$ is even, and hence $\chi(4)$ is even.

Suppose $\chi(4) \equiv 2$.
Then for $s \in \{0,1\}$ we have 
\begin{align*}
&s \cdot 2 + (2 - 2s) \cdot 3 + (s + k - 3) \cdot 4 = 4k - 6 \\ 
&\implies s\chi(2) + (2-2s)\chi(3) + (s - 3)\chi(4) + \chi(4k - 6) \nequiv 0 \pmod{4}\\
&\implies \chi(4k - 6) \nequiv 2s \pmod{4}.
\end{align*}
Therefore $\chi(4k - 6)$ is odd. We also have
\begin{align*}
&2s \cdot 1 + (k - 1 - 4s) \cdot 2 + s\cdot 2 + s\cdot 4 = 2k - 2 \\
&\implies \chi(2k-2) \nequiv 2s + 2 \pmod{4}.
\end{align*}
So $\chi(2k - 2)$ is odd as well.
But then
\begin{align*}
&(2k - 2) + 2s\cdot 1 + (k-2 - 3s)\cdot 2 + s\cdot 4 = 4k - 6 \\
&\implies \chi(2k - 2) + \chi(4k - 6) + 2s \nequiv 0 \pmod{4}.
\end{align*}
Since the left hand side is always even and $s$ ranges from 0 to 1, this is a contradiction. Therefore $\chi(4) \equiv 0 \pmod{4}$.

We have now shown that $\chi(1), \chi(2), \chi(3),$ and $\chi(4)$ are all distinct, and furthermore that without loss of generality $\chi(\alpha) \equiv \alpha \pmod{4}$ for $\alpha \le 4$. The remainder of the proof will derive a contradiction in this case. We willuse the following notation:
\begin{definition}
Let $n(\star)$ denote a sequence of 1's, 2's, 3's, and 4's summing to $n$. Furthermore, when such a \ssq appears in an expression of $\mathcal{E}$, the length of the \ssq will be exactly long enough that the equation $\mathcal{E}$ has $k-1$ terms on the left hand side.
\end{definition}
Note that the sum of the colors of the \ssq is congruent $n \pmod{4}$, and also that the sum of the \ssq is at least its length and at most 4 times its length.

Now, for all $m \ge k - 1$, we have 
$$m = m(\star).$$
We need to have $m \ge k - 1$ for this because the sequence has length $k - 1$. This implies that $\chi(m) \nequiv -m \pmod{4}$.
Similarly, for $m \ge 2k - 2$ we have
\begin{align*}
&m = k + (m-k)(\star) \\
&\implies \chi(m) \nequiv -m - \chi(k) \pmod{4}.
\end{align*}
For $m \ge 3k - 3$ we have 
\begin{align*}
&m = 2 \cdot k + (m - 2k)(\star) \\
&\implies \chi(m) \nequiv -m - 2\chi(k) \pmod{4}.
\end{align*}
Suppose that $\chi(k)$ is even. Then since we know $\chi(k) \nequiv -k \pmod{4}$, we must have $\chi(k) \equiv 2 \pmod{4}$.
So, for $s \in \{0,1\}$, since $2k - 3 - sk \ge k - 2 - s$ we have
\begin{align*}
&(2k - 2) + s \cdot k + (2k - 3 - sk)(\star) = 4k - 5 \\
&\implies \chi(4k - 5) \nequiv -\chi(2k - 2) + 2s + 3 \pmod{4}. 
\end{align*}
From the above results we know that $\chi(2k - 2) \nequiv 2, 2 - \chi(k) \pmod{4}$, so $\chi(2k - 2)$ is odd. Similarly, $\chi(4k - 5) \nequiv 1, 1 - \chi(k) \pmod{4}$, so $\chi(4k - 5)$ must be is even. But that is a contradiction, since the previous calculation shows that $\chi(4k - 5)$ is odd.
Therefore $\chi(k)$ is odd.
For $m \ge 3k - 3$ we have $\chi(m) \nequiv -m - s\chi(k) \pmod{4}$, where $s \in \{0,1,2\}$. Since $\chi(k)$ is odd, this implies that $\chi(m) \equiv -m + \chi(k) \pmod{4}$ for $m \ge 3k - 3$.

For $2k - 2 \le m \le 3k - 4$ we can compute
\begin{align*}
&m + (4k - 6 - m)(\star) = 4k - 6 \\
&\implies \chi(m) - m + 2 + \chi(4k - 6) \nequiv 0 \pmod{4}\\
&\implies \chi(m) \nequiv m - \chi(k) \pmod{4}.
\end{align*}
 For $2k - 2 \le m \le 3k - 5$ we have
 \begin{align*}
 &m + (4k - 7 -m)(\star) = 4k - 7 \\
 &\implies \chi(m) - m + 1 + \chi(4k - 7) \nequiv 0 \pmod{4}\\
 &\implies \chi(m) \nequiv m + 3  + (3 - \chi(k)) \pmod{4}\\
 &\implies \chi(m) \nequiv m + 2 - \chi(k) \pmod{4}.
 \end{align*}
 The following tables prove that for $2k -2 \le m \le 3k - 4$, $\chi(m) \equiv -m + 2 \pmod{4}$ (note that for $m=3k - 4$ we do not need the last relation):
 
$\chi(k) \equiv 1$:
 \begin{tabular}{c||c|c|c|c}
 $m$ & $-m$ & $-m - \chi(k)$ & $m - \chi(k)$ & $m + 2 - \chi(k)$ \\ \hline \hline
 0 & 0 & 3 & 3 & 1 \\ \hline 
 1 & 3 & 2 & 0 & 2 \\ \hline
 2 & 2 & 1 & 1 & 3 \\ \hline
 3 & 1 & 0 & 2 & 0 \\ 
 \end{tabular}
 \smallskip
 
 $\chi(k) \equiv 3$:
 \begin{tabular}{c||c|c|c|c}
 $m$ & $-m$ & $-m - \chi(k)$ & $m - \chi(k)$ & $m + 2 - \chi(k)$ \\ \hline \hline
 0 & 0 & 1 & 1 & 3 \\ \hline 
 1 & 3 & 0 & 2 & 0 \\ \hline
 2 & 2 & 3 & 3 & 1 \\ \hline
 3 & 1 & 2 & 0 & 2 \\ 
 \end{tabular}
 
 By process of elimination, we can see $\chi(m) \equiv -m + 2 \pmod{4}$ for $2k - 2 \le m \le 3k - 4$.
 Finally,
 \begin{align*}
 &(2k - 2) + (k) + (k-3)(\star) = 4k - 5 \\
 &\implies \chi(2k - 2) + \chi(k) + k - 3 + \chi(4k - 5) \nequiv 0 \pmod{4}\\
 &\implies 0 + \chi(k) + 1 + 1 + \chi(k) \nequiv 0 \pmod{4}\\
 &\implies 2(\chi(k) + 1) \nequiv 0 \pmod{4}.
 \end{align*} 
 Since $\chi(k)$ is odd, this is a contradiction.
\end{proof}

Finally, we present the upper bound for general $r$. The proof is similar to repeated application of the proof of Theorem~\ref{thm:uprime}.

\ucomp*

\begin{proof}
We will need the following lemma.
\begin{lemma}
\label{psum}
For positive integers $x_i$, $\sum_{i=1}^n (x_i - 1) \le \prod_{i=1}^n x_i$.
\end{lemma}
\begin{proof}
Note for positive integers $a,b$ we have $(a-1)(b-1) \ge -1 \implies a + b - 2 \le ab$. The result follows by induction.
\end{proof}
Suppose for the sake of contradiction that there is an $r$-coloring $\chi$ of the first $kr - \sum_{i=0}^{t-1}(p_i - 1) - 1$ natural numbers which does not admit an $r$-zero-sum solution to $\mathcal{E}$. By the lemma, $kr - \sum_{i=0}^{t-1}(p_i - 1) - 1 \ge kr - r$, and for the first part of the proof we will only need to use numbers up to $kr - r$. 

Since we can translate $\chi$ without affecting zero-sum sequence of length $k$, assume without loss of generality that $\chi(2) \equiv 2\chi(1) \pmod{r}$. 
First we will show that $\chi(\alpha) \equiv \alpha\chi(1) \pmod{r}$ for all $\alpha \le \lfloor \frac{3r}{8} \rfloor + 1$.
To that end, let $\alpha$ be some element of $\{2,3,\dots, \lfloor\frac{3r}{8}\rfloor \}$. We will show that $\chi(\alpha - 1) + \chi(\alpha + 1) - 2\chi(\alpha) \equiv 0 \pmod{r}$. With a simple induction argument, this implies the above claim.

Let $\nu_0 = r$, and suppose that $\chi(\alpha-1) + \chi(\alpha + 1) - 2\chi(\alpha) \nequiv 0 \pmod{r}$. Then let $\lambda_0$ be the least positive integer such that $\lambda_0(\chi(\alpha-1) + \chi(\alpha + 1) - 2\chi(\alpha) ) \equiv 0 \pmod{r}$, and let $\nu_1 = \nu_0 / \lamda_0$. 
We will use a slightly different definition of a \ssq than in the previous proof.
\begin{definition}
Let $n(\star)$ be a sequence summing to $n$ whose colors sum to $n\chi(1) \pmod{\nu_1}$. As before, when a \ssq appears in an expression of $\mathcal{E}$, its length will be such that the left hand side of the equation has exactly $k - 1$ terms. We will emphasize that the sum of the colors of the \ssq is congruent to $n\chi(1)$ mod $\nu_1$ by calling it a \emph{valid} \ssq.
\end{definition}

\begin{lemma}
\label{sl1}
If a $\star$-sequence has length $z$, where $z \ge 2\nu_1 - 1$ elements, then its sum can take on any value greater than or equal to $z$.
\end{lemma}
\begin{proof}
Suppose $n \ge z$ and write $n = (z - 2\nu_1 + 1)\cdot 1 + \nu_1\cdot y + p \cdot 2 + (\nu_1 - 1 - p)\cdot 1$, where $y \ge 1$ and $p < \nu_1$. We can compute 
\begin{align*}
&(z - 2\nu_1 + 1)\chi(1) + \nu_1\chi(y) + p\chi(2) + (\nu_1 - 1 - p)\chi(1)  \\
&\equiv (z - 2\nu_1 + 1)\chi(1) + \nu_1y\chi(1) + 2p\chi(1) + (\nu_1 - 1 - p)\chi(1)  \pmod{\nu_1} \\
&\equiv n\chi(1) \pmod{\nu_1}.
\end{align*}
\end{proof}

For $m \ge k - 1 + 2(\lambda_0 - 1)(\alpha - 1)$, and for $0 \le s_0 \le \lambda_0 - 1$, we can write
\begin{align*}
&m = (2\lambda_0 - 2 - 2s_0)\cdot \alpha + s_0\cdot (\alpha-1) + s_0 \cdot(\alpha + 1) + (m - 2\alpha\lambda_0 + 2\alpha)(\star) \\
&\implies \chi(m) + m\chi(1) + s_0(\chi(\alpha - 1) + \chi(\alpha + 1) - 2\chi(\alpha)) \nequiv 0 \pmod{\nu_0} \\
&\implies \chi(m) + m\chi(1) \nequiv 0 \pmod{\nu_1}.
\end{align*}
If $\nu_1 = 1$, we are done, since no value can be assigned to $\chi(m)$, and it is easy to check that 
\begin{align*}
k - 1 + 2(\lambda_0 - 1)(\alpha - 1) &= k - 1 + 2(r - 1)(\alpha - 1) \\
&\le k - 1 + 2(r-1)\left(\frac{3r}{8} - 1\right) \\
&= k - 1 + \frac{3r^2}{4} - 4r - \frac{3r}{4} + 4 \\
&\le \left(\frac{3}{8} + \frac{1}{r}\right)kr - \frac{11r}{4} + 3 \\
&\le kr - r,
\end{align*}
Also, if we choose the minimal $m$, we can choose $\star$-sequence to be all 1's, so we can indeed find a valid $\star$-sequence.
On the other hand, if $\nu_1 \ne 1$, let $h_1$ be the least such $m$ such that $h_1 \equiv -h_1 \pmod{\nu_1}$. Thus $h_1 \le k - 1+ 2(\lambda_0 - 1)(\alpha - 1) + \nu_1 - 1$.
Then $\chi(h_1) - h_1\chi(1) \nequiv 0 \pmod{\nu_1}$, and let $\lambda_1$ be the least positive integer such that $\lambda_1(\chi(h_1) - h_1\chi(1)) \equiv 0 \pmod{\nu_1}$. 

Suppose that for all $j < i$, $\lambda_j, \nu_j,$ and $h_j$ are defined. We will define $\lambda_i, \nu_i$, and $h_i$. First, let $\nu_i = \nu_{i-1}/\lambda_{i-1}$.
Now, for 
\begin{align*}
m &\ge \left(\sum_{j=1}^{i-1} (\lambda_j - 1)h_j\right) + (2\lambda_0 - 2)(\alpha) + \left(k - 1 - 2\lambda_0 + 2 - \left(\sum_{j=1}^{i-1}(\lambda_j - 1)\right)\right) \\
&= \left(\sum_{j=1}^{i-1} (\lambda_j - 1)(h_j - 1)\right) + k - 1 + 2(\lambda_0 - 1)(\alpha - 1),
\end{align*}
 we have
\begin{align*}
&m = \left(\sum_{j = 1}^{i-1} s_j\cdot h_j \right) + (2\lambda_0 - 2 - 2s_0) \cdot \alpha + s_0\cdot (\alpha - 1) + s_0\cdot (\alpha + 1) + \left(m - 2\lambda_0\alpha + 2\alpha - \left(\sum_{j = 1}^{i-1} s_jh_j \right)\right)(\star) \\
&\implies \chi(m) + m\chi(1) + \left(\sum_{j = 1}^{i-1} s_j(\chi(h_j) -h_j\chi(1) \right) + s_0(\chi(\alpha - 1) + \chi(\alpha + 1) - 2\chi(\alpha)) \nequiv 0 \pmod{\nu_0} \\
&\implies \chi(m) + m\chi(1) + \left(\sum_{j = 1}^{i-1} s_j(\chi(h_j) -h_j\chi(1) \right) \nequiv 0 \pmod{\nu_1}.
\end{align*}
If we let the $s_j$ vary from $0$ to $\lambda_j - 1$ for each $j$, this implies that $\chi(m) + m\chi(1) \nequiv 0 \pmod{\nu_i}$.
Let $h_i$ be the least such $m$ satisfying $h_i \equiv -h_i \pmod{\nu_i}$. Then $h_i \le \left(\sum_{j=1}^{i-1} (\lambda_j - 1)(h_j - 1)\right) + k - 1 + 2(\lambda_0-1)(\alpha - 1) + \nu_i - 1$, and we have $\chi(h_i) - h_i\chi(1) \nequiv 0 \pmod{\nu_i}$. Let $\lambda_i$ be the least positive integer such that $\lambda_i(\chi(h_i) - h_i\chi(1)) \equiv 0 \pmod{\nu_i}$.

Note that $\nu_i < \nu_{i-1}$ and $\nu_i \mid \nu_{i-1}$, so eventually we will reach an index $t$ such that $\nu_t = 1$, and this process will terminate.
Then, for
\begin{align*}
m &\ge \left(\sum_{j=1}^{t-1} (\lambda_j - 1)h_j\right) + (2\lambda_0 - 2)(\alpha) + \left(k - 1 - 2\lambda_0 + 2 - \left(\sum_{j=1}^{t-1}(\lambda_j - 1)\right)\right) \\
&= \left(\sum_{j=1}^{t-1} (\lambda_j - 1)(h_j - 1)\right) + k - 1 + 2(\lambda_0 - 1)(\alpha - 1),
\end{align*}
we have 
\begin{align*}
&m = \left(\sum_{j = 1}^{t-1} s_j \cdot h_j \right) + (2\lambda_0 - 2 - 2s_0)\cdot \alpha + s_0\cdot(\alpha - 1) + s_0\cdot(\alpha + 1) + \left(m - 2\lambda_0\alpha + 2\alpha - \left(\sum_{j = 1}^{t-1} s_jh_j \right)\right)(\star) \\
&\implies \chi(m) + m\chi(1) + \left(\sum_{j = 1}^{t-1} s_j(\chi(h_j) -h_j\chi(1) \right) + s_0(\chi(\alpha-1) + \chi(\alpha+1) - 2\chi(\alpha)) \nequiv 0 \pmod{\nu_0} \\
&\implies \chi(m) + m\chi(1) + \left(\sum_{j = 1}^{t-1} s_j(\chi(h_j) -h_j\chi(1) \right) \nequiv 0 \pmod{\nu_1}.
\end{align*}
If we let each $s_j$ vary from $0$ to $\lambda_j - 1$, this implies that $\chi(m) + m\chi(1) \nequiv 0 \pmod{\nu_t}$, which is impossible since $\nu_t = 1$.

It is necessary to check that the $\star$-sequences are sufficiently large to apply Lemma~\ref{sl1} and that the minimal $m$ are sufficiently small (smaller than $kr - r$). The worst case for both of these is the index $t$ case with all $s_j$ maximal, so it will suffice to check only that case. First we show that the $\star$-sequence has at least $2\nu_1 - 1$ elements. Recall that we have already dealt with the case of $t = 1$, so assume that $\nu_1 > 1$ and $\lambda_0 > 1$.

Now,
\begin{align*}
&(\nu_1 - 1)\left(\lambda_0 - \frac{3}{2}\right) \ge \frac{1}{2} \\
&\implies \nu_1\left(\lambda_0 - \frac{3}{2}\right) \ge \lambda_0 - 1 \\
&\implies \nu_1\lambda_0 - \lambda_0 - \frac{3}{2}\nu_1 + 1 \ge 0 \\
&\implies 2\nu_1\lambda_0 - 2\lambda_0 + 2 - 3\nu_1 \ge 0 \\
&\implies 2r - 2\lambda_0 + 2 - 3\nu_1 \ge 0 \\
&\implies k - 1 - 2\lambda_0 + 3 - 3\nu_1 \ge 0 \\
&\implies k - 1 - 2\lambda_0 + 2 - \nu_1 \ge 2\nu_1 - 1 \\
&\implies k - 1 - 2\lambda_0 + 2 - \prod_{j=1}^{t-1}\lambda_j \ge 2\nu_1 - 1 \\
&\implies k - 1 - 2\lambda_0 + 2 - \sum_{j=1}^{t-1}(\lambda_j - 1) \ge 2\nu_1 - 1,
\end{align*}
which shows that the $\star$-sequence is long enough to apply Lemma~\ref{sl1}.

It remains to show that 
$$\left(\sum_{j=1}^{t-1} (\lambda_j - 1)(h_j - 1)\right) + k - 1 + 2(\lambda_0 - 1)(\alpha - 1) \le kr - r.$$
To that end, let $k = Ar$ and let 
$$V_i = \left(\sum_{j=1}^{i-1} (\lambda_j - 1)(h_j - 1)\right) + k - 1 + 2(\lambda_0 - 1)(\alpha - 1).$$
We can compute
\begin{align*}
V_{i+1} &= V_i + (\lambda_i - 1)(h_i - 1) \\
&\le V_i + (\lambda_i - 1)\left(\left(\sum_{j=1}^{i-1}(\lambda_j-1)(h_j -1)\right) + k - 1 + 2(\lambda_0 - 1)(\alpha - 1) + \nu_i - 2 \right) \\
&= V_i + (\lambda_i - 1)(V_i + \nu_i - 2) \\
&= \lambda_iV_i + (\lambda_i\nu_i - \nu_i - 2\lambda_i + 2).
\end{align*}

Now, since $\nu_i = \nu_{i-1} / \lambda_{i-1}$, we can check that $\prod_{j=q}^{t-1}\lambda_j = \nu_q$. Thus we have
\begin{align*}
V_t &\le V_1\nu_1 + \sum_{j=1}^{t-1}\nu_{j+1}(\lambda_j\nu_j - \nu_j - 2\lambda_j + 2) \\
&= V_1\nu_1 + \sum_{j=1}^{t-1}(\nu_j^2 - \nu_j\nu_{j+1} - 2\nu_j + 2\nu_{j+1}) \\
&\le V_1\nu_1 + \sum_{j=1}^{t-1}(\nu_j^2 - \nu_{j+1}^2 - 2\nu_j + 2\nu_{j+1}) \\
&\le V_1\nu_1 + \nu_1^2 - 1 + 2 - 2\nu_1 \\
&= k\nu_1 - \nu_1 + 2\lambda_0\alpha\nu_1 - 2\alpha\nu_1 - 2\lambda_0\nu_1 + 2\nu_1 + (\nu_1 - 1)^2 \\
&= Ar\nu_1 + \nu_1 + 2\alpha r - 2\alpha\nu_1 - 2r + (\nu_1 - 1)^2 \\
&= r(A\nu_1 + 2\alpha - 2) - 2\alpha\nu_1 + \nu_1 + (\nu_1 - 1)^2 \\
&\le r\left(\frac{Ar}{2} + 2\alpha - 2\right) - 2\alpha + \frac{r}{2} + \frac{r^2}{4} - r + 1 \\
&= \left(\frac{2A + 1}{4}\right)r^2 + \left(2\alpha - \frac{5}{2}\right)r - 2\alpha + 1.
\end{align*}
Finally, since $\alpha \le \frac{3r}{8}$ and $A \ge 2$ we have
\begin{align*}
\frac{3}{4}r &\ge 2\alpha - \frac{3}{2} \\
 \protect \frac{2A - 1}{4}r &\ge 2\alpha - \frac{3}{2} \\
 \protect \left(\frac{2A - 1}{4}\right)r^2 &\ge \left(2\alpha - \frac{3}{2}\right)r - 2\alpha + 1 \\
 Ar^2 - r &\ge \left(\frac{2A + 1}{4}\right)r^2 + \left(2\alpha - \frac{5}{2}\right)r - 2\alpha + 1,
\end{align*}
as desired.
Thus $\chi(\alpha - 1) + \chi(\alpha + 1) - 2\chi(\alpha) \equiv 0 \pmod{r}$ for $\alpha \le \lfloor \frac{3r}{8} \rfloor$, and by induction $\chi(\alpha) - \alpha\chi(1) \equiv 0 \pmod{r}$ for $\alpha \le \lfloor \frac{3r}{8} \rfloor + 1$.

For the second part of the proof, we will also need the following four special cases:
\begin{description}
\item[$\mathbf{A = 2, r = 9, \alpha = 4:}$]
In this case we have $Ar^2 - r  = 153$, and  $\left(\frac{2A + 1}{4}\right)r^2 + \left(2\alpha - \frac{5}{2}\right)r - 2\alpha + 1 = 101.25 \le 153$. So $\chi(5) \equiv 5\chi(1) \pmod{r}$ also.

\item[$\mathbf{A = 2, r = 10, \alpha = 4:}$]
In this case we have $Ar^2 - r = 190$, and $\left(\frac{2A + 1}{4}\right)r^2 + \left(2\alpha - \frac{5}{2}\right)r - 2\alpha + 1 = 121 \le 190$. So $\chi(5) \equiv 5\chi(1) \pmod{r}$ also.

\item[$\mathbf{A = 2, r = 6, \alpha = 3:}$]
In this case we have $Ar^2 - r = 66$, and $\left(\frac{2A + 1}{4}\right)r^2 + \left(2\alpha - \frac{5}{2}\right)r - 2\alpha + 1 = 43 \le 66$. So $\chi(4) \equiv 4\chi(1) \pmod{r}$ also.

\item[$\mathbf{A = 2, r = 8, \alpha = 4:}$]
In this case we have $Ar^2 - r = 120$, and $\left(\frac{2A + 1}{4}\right)r^2 + \left(2\alpha - \frac{5}{2}\right)r - 2\alpha + 1 = 85 \le 120$. So $\chi(5) \equiv 5\chi(1) \pmod{r}$ also.
\end{description}

Now we will show that $\chi(k) - k\chi(1) \equiv 0 \pmod{r}$. For the remainder of the proof, we will use a slightly different definition of a $\star$-sequence:

\begin{definition}
Let $n(\star)$ be a sequence summing to $n$ whose colors sum to $n\chi(1) \pmod{r}$. As before, when a \ssq appears in an expression of $\mathcal{E}$, its length will be such that the left hand side of the equation has exactly $k - 1$ terms. We will again emphasize that the sum of the colors of the \ssq is congruent to $n\chi(1)$ mod $r$ by calling it a \emph{valid} \ssq.
\end{definition}

\begin{lemma}
\label{sl2}
If a $\star$-sequence has $z \ge r + 4$ elements in it, it can take on any value greater than or equal to $z$.
\end{lemma}

\begin{proof}
A sequence of $l$ elements of $\{1,2,\dots,\lfloor \frac{3}{8} \rfloor + 1\}$ can sum to any of $l\lfloor \frac{3}{8} \rfloor + 1$ distinct values. If $l\lfloor \frac{3}{8} \rfloor + 1 \ge r$, then this sum can attain any value modulo $r$. Since $r \ge 6$, if $l \ge 4$ we have 
\begin{align*}
l\left\lfloor \frac{3}{8} \right\rfloor + 1 &\ge 4\left(\frac{3r}{8} - 1\right) + 1 \\
&\ge \frac{3r}{2} - 3 \\
&\ge r + \frac{r}{2} - 3 \\
& \ge r,
\end{align*}
as desired. Then any value $n$ greater than the length $z$ of the \ssq can be expressed as $r(y)$ plus the sum of at least $4$ values from the set $\{1,2,\dots,\lfloor \frac{3}{8} \rfloor + 1\}$, where $y$ is some positive integer. Since $r\chi(y) \equiv ry\chi(1) \equiv 0 \pmod{r}$, the sum of the colors of this sequence is indeed $n \pmod{r}$.
\end{proof}

Let $h_0 = k$, $\nu_0 = r$, and suppose that $\chi(h_0) - h_0\chi(1) \nequiv 0 \pmod{\nu_0}$. Then let $\lambda_0$ be the least positive integer such that $\lambda_0(\chi(h_0) - h_0\chi(1)) \equiv 0 \pmod{\nu_0}$. 
Let $\nu_1 = \nu_0 / \lambda_0$. 
Now consider for some integers $m$, $s_0$:
\begin{align*}
&m = s_0 \cdot h_0 + (m - s_0h_0)(\star) \\
&\implies \chi(m) + m\chi(1) + s_0(\chi(h_0) - h_0\chi(1)) \nequiv 0 \pmod{\nu_0}. \\
\end{align*}
If $s_0$ ranges from $0$ to $\lambda_0 - 1$, this implies that 
$$\chi(m) + m\chi(1) \nequiv 0 \pmod{\nu_1}.$$
Furthermore, if $\nu_1 \mid m$, we have $\chi(m) + m\chi(1) \equiv \chi(m) - m\chi(1) \nequiv 0 \pmod{\nu_1}$.
The least value of $m$ for which we can allow $s_0$ to range from $0$ to $\lambda_0 - 1$ is 
$$m = (\lambda_0 - 1)h_0 + (k - \lambda_0) = (\lambda_0 - 1)(h_0 - 1) + k - 1.$$ 
Therefore there is some $y_1 \le \lambda_0 - 1$ such that if $h_1 = (\lambda_0 - 1)h_0 + (k - \lambda_0) + y_1$, then $\nu_1 \mid h_1$, and we have $\chi(h_1) - h_1\chi(1) \nequiv 0 \pmod{\nu_1}$.  Let $\lambda_1$ be the least positive integer such that $\lambda_1(\chi(h_1) - h_1\chi(1)) \equiv 0 \pmod{\nu_1}$.

Suppose that for $j < i$, $\lambda_j, \nu_j,$ and $h_j$ have already been defined, and that $\nu_{i-1} \ne 1$. We will define $\lambda_i, \nu_i,$ and $h_i$. First let $\nu_i = \nu_{i - 1} / \lambda_{i - 1}$. 
Then for some values of $m$, $s_j$ we have
\begin{align*}
&m = \left(\sum_{j=0}^{i-1} s_j \cdot h_j\right) + \left(m - \left(\sum_{j=0}^{i-1} s_jh_j\right)\right)(\star) \\
&\implies \chi(m) + m\chi(1) + \left( \sum_{j=0}^{i-1} s_j(\chi(h_j) - h_j\chi(1)) \right) \nequiv 0 \pmod{r}.
\end{align*}
If we allow the $s_j$ to range from $0$ to $\lambda_j - 1$, the expression $\left( \sum_{j=0}^{i-1} s_j(\chi(h_j) - h_j\chi(1)) \right)$ can take on the value of any multiple of $\nu_i \pmod{r}$. Thus we have $\chi(m) + m\chi(1) \nequiv 0 \pmod{\nu_i}$. 

The least value of $m$ for which we can produce a valid $\star$-sequence is 
$$m = \left(\sum_{j=0}^{i-1} (\lambda_j - 1)h_j \right) + \left(k - 1 - \sum_{j=0}^{i-1} (\lambda_j - 1)\right) = \left(\sum_{j=0}^{i-1} (\lambda_j - 1)(h_j - 1) \right) + k - 1.$$ 
(If $m$ is smaller than this, the length of the \ssq would be less than its sum.) Thus there is some value $y_i \le \nu_i - 1$, such that if $h_i = \left(\sum_{j=0}^{i-1} (\lambda_j - 1)(h_j - 1) \right) + k - 1 + y_i$, then $\nu_i \mid h_i$, and therefore $\chi(h_i) + h_i\chi(1) \equiv \chi(h_i) - h_i\chi(1) \nequiv 0 \pmod{\nu_i}$. Then let $\lambda_i$ be the least positive integer such that $\lambda_i(\chi(h_i) - h_i\chi(1)) \equiv 0 \pmod{\nu_i}$.

Note that $\nu_{i} \mid \nu_{i-1}$ and also $\nu_i < \nu_{i-1}$, so there is some index $t$ at which $\nu_t = 1$ and this process terminates. Then for some values of $m,s_j$ we can write
\begin{align*}
&m = \left(\sum_{j=0}^{t-1} s_j \cdot h_j\right) + \left(m - \left(\sum_{j=0}^{t-1} s_jh_j \right)\right)(\star) \\
&\implies \chi(m) + m\chi(1) +  \left(\sum_{j=0}^{t-1} s_j(\chi(h_j) - h_j\chi(1))\right) \nequiv 0 \pmod{r}. \\
\end{align*}
As before, if each $s_j$ is allowed to range from $0$ to $\lambda_j - 1$, this implies that $\chi(m) + m\chi(1) \nequiv 0 \pmod{\nu_t}$. But $\nu_t = 1$, so this is impossible. Let $m_t$ be the least value of the above for which it is possible to have a valid $\star$-sequence. We compute
\begin{align*}
m_t &= \left(\sum_{j=0}^{t-1}(\lambda_j - 1)h_j\right) + \left( k - 1 - \sum_{j=0}^{t-1} (\lambda_j - 1) \right) \\
&=  \left(\sum_{j=0}^{t-1}(\lambda_j - 1)(h_j - 1) \right) + k - 1.
\end{align*}
In order to finish the proof, we need to check that $m_t \le kr - \sum_{i=0}^{t-1}(p_i - 1) - 1$, and also check that we can apply Lemma~\ref{sl2} to produce a valid $\star$-sequence. To that end, let 
$$V_i = \left(\sum_{j=0}^{i-1}(\lambda_j - 1)(h_j - 1) \right) + k - 1.$$
We can compute
\begin{align*}
V_{i+1} &= V_i + (\lambda_i - 1)(h_i - 1) \\
&= V_i + (\lambda_i - 1)\left(\sum_{j=0}^{i-1} (\lambda_j - 1)(h_j - 1) + k -1 + y_i - 1 \right) \\
&= V_i + (\lambda_i - 1)(V_i + y_i - 1) \\
&= \lambda_iV_i + (\lamda_i - 1)(y_i - 1) \\
&= \lambda_iV_i + (\lambda_iy_i - y_i - \lambda_i + 1). 
\end{align*}
Since $\nu_i = \nu_{i-1}/\lambda_{i-1},$ we have $\prod_{j=q}^{t-1}\lambda_j = \nu_q$. Therefore
\begin{align*}
V_t &\le V_0\nu_0 + \sum_{j=1}^{t}\nu_j(\lambda_{j-1}y_{j-1} - y_{j-1} - \lamda_{j-1} + 1) \\
&= kr - r + \sum_{j=1}^{t} (\nu_j\lambda_{j-1}y_{j-1} - \nu_jy_{j-1} - \nu_j\lambda_{j-1} + \nu_j) \\
&= kr - r + \sum_{j=1}^{t} (\nu_{j-1}y_{j-1} - \nu_jy_{j-1} - \nu_{j-1} + \nu_j) \\
&= kr - r + \sum_{j=1}^{t} (\nu_{j-1}y_{j-1} - \nu_jy_{j-1}) - \nu_0 + \nu_t \\
&= kr - r + \sum_{j=1}^{t-1} (\nu_j(y_{j} - y_{j-1})) - y_{t-1} + 1.
\end{align*}
Now observe that $V_i + y_i = h_i$, and that 
$$h_i \equiv \lambda_{i-1}(h_{i-1} - y_{i-1}) + (\lambda_{i-1} - 1)(y_{i-1} - 1) + y_i \equiv y_i - y_{i-1} - \lambda_{i-1} + 1 \pmod{\nu_i}.$$
Therefore $y_i - y_{i-1} \equiv \lambda_{i-1} - 1 \pmod{\nu_i}$, and in particular $y_i - y_{i-1} \le \lambda_{i-1} - 1$, and $y_{t-1} \le 1 + \sum_{j=0}^{t-2} (\lambda_j - 1)$.
This allows us to finish the above calculation:
\begin{align*}
&= kr - r + \sum_{j=1}^{t-1} (\nu_j(y_{j} - y_{j-1})) - y_{t-1} + 1 \\
&\le kr - r + \sum_{j=1}^{t-1} (\nu_j(\lamba_{j-1} - 1)) - \sum_{j=0}^{t-2} (\lambda_j - 1)\\
&= kr - r + \sum_{j=1}^{t-1} (\nu_{j-1}- \nu_j) - \sum_{j=0}^{t-2} (\lambda_j - 1) \\
&= kr - \nu_{t-1}  - \sum_{j=0}^{t-2} (\lambda_j - 1) \\
&= kr -  \sum_{j=0}^{t-1} (\lambda_j - 1) - 1.
\end{align*}
By Lemma~\ref{psum}, this value is at most $kr - \sum_{i=0}^{t-1}(p_i - 1) - 1$, as desired.

We also must check that the length of the star sequence is always at least $r + 4$, so that we can apply Lemma~\ref{sl2} and show that there is a valid $\star$-sequence.
In the worst case, all of the $s_j$ attain their maximum value, $\lambda_j - 1$. Then the $\star$-sequence has length
\begin{align*}
&k - 1 - \sum_{j=0}^{t-1}(\lambda_j - 1) \\
\ge &2r - 1 - \sum_{j=0}^{t-1}(\lambda_j - 1)\nu.
\end{align*}
Thus we need to show
$$\sum_{j=0}^{t-1}(\lamba_j - 1) \le r - 5.$$
In fact, this is not always true. We divide into $2$ cases:
\begin{description}
   \item[Case 1: $t = 1$.]
   We must show that it is possible to write $m = s_0k + (m - s_0k)(\star)$ for $0 \le s_0 \le r - 1$ for some 
   $m \le kr - r$ (since $t = 1 \implies V_t = kr - r$). In fact, we will take $m$ to be $kr - r$. 
   First note that the value of the $\star$-sequence is always at least its length, since in the worst case 
   ($s_0 = \lambda_0 - 1$) the length of the \ssq is $k - r$, and its value is $kr - r - (kr - k)= k - r$.
   Suppose $k \ge 3r$. 
   Then the length of the $\star$-sequence is at least $2r \ge r+4$, so we can apply the lemma
   and the \ssq is always valid.
   On the other hand, if $k = 2r$, the star-sequence is not always long enough to apply the lemma. The length 
   of the \ssq is equal to $k - 1 - s_0 = 2r - 1 - s_0$, so when $r - 4 \le s_0 \le r - 1$ we must find a valid $\star$-sequence. 
   \begin{itemize}
   	\item The case $s_0 = r - 1$ is easy since we can take the sequence of all 1's.
	\item If $s_0 = r - 2$, we must find $r + 1$ elements summing to $2k - r = 3r$. Since $r \ge 6$, we have 
 	$3 \ge \lfloor \frac{3r}{8} 	\rfloor + 1$, so we can just take $r - 1$ 3's, one 2, and one 1.
 	 \item If $s_0 = r - 3$, we must find $r + 2$ elements summing to $3k - r = 5r$. If $r = 6$ or $8$, 
	 then note that $(\lfloor \frac{3}{8} 	\rfloor + 1)(r + 2) \ge 5r$, so this is possible. If $r \ge 12$, note that
	 $\left(\frac{3}{8}r\right)(r + 2) \ge 5r$, so this is again possible. We will ignore the cases $r = 7$ 
	 and $r = 11$ since they are solved by the previous theorem on primes. 
	 In the cases $r = 9, r=10$, we can use our special cases from the first part of the proof, and since 
	 $5(r+2) \ge 5r$, it will be possible to write $5r$ as the sum of $r+2$ numbers, each at most $5$.
 	 \item If $s_0 = r - 4$, we must find $r + 3$ elements summing to $4k - r = 7r$. Take $r$ copies of $6$, 
 	  and then find $3$ elements which are at most $\lfloor \frac{3r}{8} 	\rfloor + 1$ that sum to $r$.  
	  It is easy to check that this is always possible when $r \ge 6$.
   \end{itemize}
   \item[Case 2: $t \ge 2$.] 
   By a previous lemma, 
   $\sum_{j=0}^{t-1}(\lamba_j - 1) \le \lambda_0 - 1 + \frac{r}{\lambda_0}$. But since $t \ge 2$, 
   $\lambda_0 < r$ implies  $2 \le \lambda_0 \le \frac{r}{2}$. By convexity $\lambda_0 - 1 + \frac{r}{\lambda_0}$ 
   is maximized at the extremes $\lambda_0 = 2$ and $\lambda_0 = \frac{r}{2}$. In these cases we have
   $\lambda_0 - 1 + \frac{r}{\lambda_0} = \frac{r}{2} + 1$. If $r \ge 12$, we are done, since $\frac{r}{2} + 1 \le r - 5$. As in the $t = 1$ case, if 
   $k \ge 3r$, we are done. Also as in the $t = 1$ case, we will ignore $r = 7, r = 11$ since they are solved by a 
   previous theorem. So we have 4 subcases:
   \begin{itemize}
   	\item The case $r = 6$. Then $t = 2$ and $\{\lambda_0, \lambda_1\} = \{2,3\}$ in some order. 
	Then the \ssq has at least $11 - 2 - 1 = 8$ elements in it. Since for all $\alpha \le 4$, we have 
	$\chi(\alpha) \equiv \alpha\chi(1) \pmod{r}$, the proof of Lemma~\ref{sl2} can be modified
	to only require a \ssq of length $r + \lceil\frac{r - 1}{3}\rceil = 8$ elements. 
	\item The case $r = 8$. Similarly to the $r = 6$ case, at worst we have $15 - 3 - 2 = 10$ elements in 
	the $\star$-sequence. As before we can modify the proof of Lemma~\ref{sl2} to only require a 
	a length of $r +  \lceil\frac{r - 1}{4}\rceil = 10$ elements.
	\item The case $r = 9$. Similarly to the previous cases, we can check that
	the \ssq has length at least $17 - 2 - 2 = 13 \ge r + 4$, so Lemma~\ref{sl2} applies.
	\item The case $r = 10$. Similarly to the previous cases, we can check that 
	the \ssq has length at least $19 - 4 - 1 = 14 \ge r + 4$, so Lemma~\ref{sl2} applies.
   \end{itemize}
\end{description}

The above cases show that we can always find a valid \ssq with the desired length and sum, which completes the proof that there is a contradiction unless $\chi(k) - k\chi(1) \equiv 0 \pmod{r}$. Of course, we also have
$$(k-2)\cdot 1 + 2 = k,$$
so
$$(k-2)\chi(1) + \chi(2) + \chi(k) \nequiv 0 \pmod{r}.$$
Since $\chi(2) \equiv 2\chi(1) \pmod{r}$, this leads to a contradiction immediately, and the proof is complete.

\end{proof}

\section{Lower Bounds}
Here we present two shorter proofs of lower bounds on $\sz{k}{r}$. Recall the statement of Theorem~\ref{thm:lodd}.

\lodd*

\begin{proof}
We will show how to construct a coloring $\chi : \{1,2,\dots,kr-r-1\} \longrightarrow \mathbb{Z} / r\mathbb{Z}$ with no $r$-zero-sum solution to $\mathcal{E}$.
The coloring will have the following properties for $1 \le \alpha \le r$:
\begin{itemize}
\item For $m \le \alpha k  - \alpha, \chi(m) \in \{m, m + 2, m + 4, \dots, m + 2(\alpha- 1)\}.$ 
\item For $m \ge \alpha k - \alpha, \chi(m) \not\in \{-m, -m - 2, -m - 4, \dots, -m - 2(\alpha - 1) \}.$
\end{itemize}
Suppose there is such a coloring, and that $\sum_{i=1}^{k-1} x_i = x_k$. Suppose the sequence $x_1, \dots, x_{k-1}$ contains $b_\alpha$ elements $x_i$ such that $(\alpha-1)k - (\alpha - 1) < x_i \le \alpha k - \alpha$.
Then $\sum_{i=1}^{k-1} \chi(x_i) \equiv x_k + 2\gamma \pmod{r}$, where $\gamma$ is some integer such that $0 \le \gamma \le \sum_{\alpha=1}^r (\alpha - 1)b_\alpha$.

Furthermore, $x_k \ge k - 1 + \sum_{\alpha=1}^r b_\alpha (\alpha - 1)(k - 1) = \left(1 + \sum_{\alpha=1}^r (\alpha - 1)b_\alpha \right)(k - 1)$. Thus $\chi(x_k) \nequiv -x_k - 2\gamma'$ for any $0 \le \gamma' \le \sum_{\alpha=1}^r (\alpha - 1)b_\alpha$.

Suppose for the sake of contradiction that $\chi(x_k) + x_k + 2\gamma \equiv 0 \pmod{r}$. Then $\chi(x_k) \equiv -x_k - 2\gamma \pmod{r}$. But this is impossible since $\chi(x_k) \nequiv -x_k - 2\gamma'$, and the bounds on $\gamma$ and $\gamma'$ are the same.

Now we will show that such a coloring exists. If $m \ne \alpha k - \alpha$ for some $\alpha$, then the set of allowed colors of $m$ is larger than the set of disallowed colors, so we can simply pick one of the allowed colors. Note that this uses the fact that $\alpha \le r$. 
On the other hand, if $m = \alpha k - \alpha$, the sets of permitted and forbidden colors are the same size. We must show that they are not equal for $\alpha < r$.
Suppose for the sake of contradiction that 
$$\{-\alpha, -\alpha + 2, -\alpha + 4, \dots, -\alpha + 2(\alpha - 1) \} \equiv \{\alpha, \alpha - 2, \alpha - 4, \dots, \alpha - 2(\alpha -1) \} \pmod{r}.$$
Since $\alpha < r$, both of these sets have at most $r$ elements in them, and the elements all differ by $2$'s. So, if the two sets are the same, it must be the case that $-\alpha \equiv \alpha - 2(\alpha - 1) \pmod{r}$. 
But that implies $0 \equiv 2 \pmod{r}$, which is a contradiction. Therefore $\sz{k}{r} > kr - r - 1$.
\end{proof}

There is a similar construction when $r$ is even, but it is no longer useful to use sequences with common difference two.

\leven*

\begin{proof}
Again we will show how to construct a coloring $\chi: \{1, 2, \dots, kr - r - 2 \} \longrightarrow \mathbb{Z}/r\mathbb{Z}$ with no $r$-zero-sum solution to $\mathcal{E}$. The coloring will have the following properties for $1 \le \alpha \le r - 2$:
\begin{itemize}
\item For $m \le \alpha k - \alpha, \chi(m) \in \{m, m + 1, m + 2, \dots, m + (\alpha - 1) \}$.
\item For $m \ge \alpha k - \alpha, \chi(m) \not\in \{-m, -m - 1, -m - 2, \dots, -m - (\alpha - 1) \}$.
\end{itemize}
Additionally, we have the following properties:
\begin{itemize}
\item For $m \le (r-1)k - (r - 1) - 1, \chi(m) \in \{m, m+1, m+2, \dots, m + (r- 2) \}$.
\item For $m \ge (r-1)k - (r-1), \chi(m) \not\in \{-m, -m - 1, -m -2, \dots, -m - (r - 2) \}$.
\end{itemize}
First we show that these properties are sufficient. Suppose for the sake of contradiction that $\chi$ satisfies the above properties and that $x_1, \dots , x_k$ is an $r$-zero-sum solution to $\mathcal{E}$. As before, for $1 \le \alpha \le r - 2$, let $b_\alpha$ be the number of elements of $x_1, \dots, x_{k-1}$ such that $(\alpha - 1)k  - (\alpha - 1) < x_i \le \alpha k - \alpha$. Additionally, let $b_{r-1}$ be the number of $x_i$ such that $(r-2)k - (r-2) < x_i \le (r-1)k - (r-1) - 1$. Note that none of $x_1, \dots, x_{k-1}$ can be be greater than or equal to $(r-1)k - (r-1)$, or $x_k$ would exceed $kr - r - 2$.

Now, $\sum_{i=1}^{k-1}\chi(x_i) \equiv x_k + \gamma$, where $\gamma$ is some integer such that $0 \le \gamma \le \sum_{\alpha=1}^{r-1}(\alpha - 1)b_\alpha$.

Furthermore, $x_k \ge k - 1 + \sum_{\alpha = 1}^{r-1}b_\alpha(\alpha - 1)(k - 1) = (1 + \sum_{\alpha=1}^{r-1} (\alpha - 1)b_\alpha)(k-1)$. Thus $\chi(x_k) \nequiv -x_k - \gamma' \pmod{r}$, where $0 \le \gamma' \le \sum_{\alpha=1}^{r-1}(\alpha - 1)b_\alpha$. Suppose that $\chi(x_k) + x_k + \gamma \equiv 0 \pmod{r}$. Then $\chi(x_k) \equiv -x_k - \gamma \pmod{r}$. But that is impossible since $\chi(x_k) \nequiv -x_k - \gamma'$, and $\gamma'$ has the same bounds as  $\gamma$.

Now we show that such a coloring exists. As in the previous proof, the set of permitted colors is larger than the set of forbidden colors except when $m = \alpha k - \alpha$ for some $\alpha \le r - 2$. Suppose for some such $\alpha$ the sets of permitted and forbidden colors are the same $\pmod{r}$. That is, suppose
$$\{-\alpha, -\alpha + 1, -\alpha + 2, \dots, -\alpha + (\alpha - 1) \} \equiv \{\alpha, \alpha - 1, \alpha - 2, \dots, \alpha - (\alpha - 1)\}. $$
Then $-\alpha \equiv \alpha - (\alpha - 1) \implies \alpha \equiv -1 \pmod{r}$. But that is impossible since $\alpha < r - 1$. Therefore $\sz{k}{r} > kr - r - 2$.
\end{proof}

\section{Conclusion and Remaining Questions}
The above theorems show that $\sz{k}{r} = kr - r$ whenever $r$ is an odd prime and $k > r$. Additionally, we have shown that $\sz{k}{4} = 4k - 5$ when $k > 4$, so Robertson's first two questions have been answered.
Robertson's fourth question, regarding the order of $\sz{k}{k}$ still remains unresolved.
Additionally we ask the following questions:
\begin{itemize}
\item Is the bound given by the summation in Theorem 3 always tight?
\item For most objects in Ramsey Theory, it seems very difficult to find upper and lower bounds which are close to each other. On the other hand, it appears to be much easier to find close bounds on the zero-sum Schur numbers than on the ordinary Schur numbers. Similar zero-sum variants can be defined for many other objects of Ramsey Theory. What more can be said about the zero-sum analogues of\dots
\begin{itemize}
\item Ramsey Numbers?
\item Van der Waerden Numbers? (see~\cite{robertsonvdw})
\item Rado Numbers? (see~\cite{brown})
\end{itemize}
\end{itemize}

\section{Acknowledgements}
This work was conducted at the 2018 Duluth REU. The author would like to thank Joe Gallian for organizing the REU and suggesting the problem, as well the advisors Levent Alpoge, Aaron Berger, and Colin Defant for their help and support. Additionally, the author would like to thank Mitchell Lee for his insightful comments. The Duluth REU is supported by NSF/DMS grant 1650947 and NSAgrant H98230-18-1-0010, and by the University of Minnesota Duluth.


\begin{thebibliography}{}
\bibitem{robertson}
A. Robertson, Zero-sum generalized Schur numbers, {arXiv}:1802.03382, Feb. 2018.

\bibitem{robertson2}
A. Robertson, B. Roy, and S. Sarkar, The Determination of 2-color zero-sum generalized Schur numbers, {arXiv}:1803.00861, Mar. 2018.

\bibitem{robertsonvdw}
A. Robertson, Zero-sum analogues of van der Waerden's Theorem on arithmetic progressions, {arXiv}:1802.03387, Feb. 2018.

\bibitem{brown}
N. Brown, On Zero-Sum Rado Numbers for the Equation $ax_1 + x_2 = x_3$, Masters thesis, South Dakota State University, 2017.

\bibitem{schur}
Schur, I. ``\"{U}ber die Kongruenz $x^m+y^m=z^m \pmod{p}$" Jahresber. \textit{Deutsche Math.-Verein}. \textbf{25}, 114-116, 1916.

\bibitem{exoo}
Exoo, G. ``A Lower Bound for Schur Numbers and Multicolor Ramsey Numbers of $K_3$." \textit{Electronic J. Combinatorics 1}, No. 1, R8, 1-3, 1994. 

\bibitem{egz}
P. Erd\H{o}s, A. Ginzberg, and A. Ziv, Theorem in additive number theory, \textit{Bulletin Research Council Israel} \textbf{10F}
(1961-2), 41-43.

\bibitem{zerosum}
Y. Caro, Zero-sum problems --- a survey, \textit{Discrete Math}. \textbf{152} (1996), 93-113.
\end{thebibliography}
\end{document}